\documentclass[12pt,reqno]{amsart}
\usepackage{amssymb,amsthm}
\usepackage{amsfonts,cmtiup,comment,stmaryrd}
\usepackage{mathrsfs,cmtiup}
\makeatletter
\def\blfootnote{\xdef\@thefnmark{}\@footnotetext}
\makeatother
\newtheorem{theorem}{Theorem}[section]
\newtheorem{lemma}[theorem]{Lemma}
\newtheorem{proposition}[theorem]{Proposition}
\newtheorem{corollary}[theorem]{Corollary}
 
\theoremstyle{definition}
\newtheorem{example}[theorem]{Example}

\newtheorem*{definition*}{Definition}
\numberwithin{equation}{section}

\begin{document}

\title{Rank type conditions on commutators in finite groups}

\author{Cristina Acciarri}
\address{C.~Acciarri: Dipartimento di Scienze Fisiche, Informatiche e Matematiche, Universit\`a degli Studi di Modena e Reggio Emilia, Via Campi 213/b, I-41125 Modena, Italy}
\email{cristina.acciarri@unimore.it}

\author{Robert M. Guralnick}
\address{Robert M. Guralnick: Department of Mathematics, University of
Southern California, Los Angeles, CA90089-2532, USA}
\email{guralnic@usc.edu}

\author{Evgeny Khukhro}
\address{E. I. Khukhro: Charlotte Scott Research Centre for Algebra, University of Lincoln, U.K.}
\email{khukhro@yahoo.co.uk}

\author{Pavel Shumyatsky}

\address{P. Shumyatsky: Department of Mathematics, University of Brasilia, DF~70910-900, Brazil}
\email{pavel@unb.br}

\thanks{The first author is member of ``National Group for Algebraic and Geometric Structures, and Their Applications'' (GNSAGA–INdAM). The second author was partially supported by NSF grant DMS 1901595 and Simons Foundation Fellowship 609771. The third author was partially supported by the International Center for Mathematics at SUSTech in Shenzhen. The fourth author was partially supported by FAPDF and CNPq.}

\begin{abstract}
 For a subset $S$ of a group $G$, let $I_G(S)$ denote the set of commutators $[g,s]=g^{-1}g^s$, where $g\in G$ and $s\in S$, so that $[G,S]$ is the subgroup generated by $I_G(S)$. We prove that if $G$ is a $p$-soluble finite group with a Sylow $p$-subgroup $P$ such that any subgroup generated by a subset of $I_G(P)$ is $r$-generated, then $[G,P]$ has $r$-bounded rank. We produce examples showing that such a result does not hold without the assumption of $p$-solubility. Instead, we prove that if a finite group $G$ has a Sylow $p$-subgroup $P$ such that (a) any subgroup generated by a subset of $I_G(P)$ is $r$-generated, and (b) for any $x\in I_G(P)$, any subgroup generated by a subset of $I_G(x)$ is $r$-generated, then $[G,P]$ has $r$-bounded rank. We also prove that if $G$ is a finite group such that for every prime $p$ dividing $|G|$, for any Sylow $p$-subgroup $P$, any subgroup generated by a subset of $I_G(P)$ can be generated by $r$ elements, then the derived subgroup $G'$
has $r$-bounded rank. As an important tool in the proofs, we prove the following result, which is also of independent interest: if a finite group $G$ admits a group of coprime automorphisms $A$ such that any subgroup generated by a subset of $I_G(A)$ is $r$-generated, then the rank of $[G,A]$ is $r$-bounded.
\end{abstract}

\maketitle

\section{Introduction}

By the rank of a finite group $G$ we mean the least positive integer $r$ such that every subgroup of $G$ can be generated by $r$ elements. (This parameter is also called the sectional rank, or the Pr\"ufer rank.) Bounds for the rank of finite groups or their subgroups impose strong restrictions on their structure. Results of this kind also have applications in wider classes of groups, such as profinite groups or locally finite groups. For example, conditions on the ranks are central in the theory of powerful $p$-groups developed by A.~Lubotzky and A.~Mann \cite{LM}, as well as in several important applications of this theory for profinite and residually finite groups. In representation theory of finite groups and its applications, the ranks are related to the dimensions of linear spaces arising as elementary abelian sections of the group. For example, Zassenhaus' theorem on soluble linear group implies a bound for the Fitting height of a finite soluble group of given rank. Therefore bounding the ranks is an important avenue of research in group theory.

In this paper we study the ranks of subgroups generated by commutators of elements of finite groups. For a subset $S$ of a group $G$, let $I_G(S)$ denote the set of commutators $[g,s]=g^{-1}g^s$, where $g\in G$ and $s\in S$, so that $[G,S]$ is the subgroup generated by $I_G(S)$.

One of our main results is the following `local--global' theorem about ranks in $p$-soluble finite groups. Henceforth we write ``$(a,b,c\dots)$-bounded'' to abbreviate ``bounded above by some function depending only on the parameters $a,b,c\dots$".

\begin{theorem}\label{second}
Let $p$ be a prime, $r$ a positive integer, $G$ a $p$-soluble finite group, and $P$ a Sylow $p$-subgroup of $G$. Suppose that any subgroup generated by a subset of $I_G(P)$ can be generated by $r$ elements. Then $[G,P]$ has $r$-bounded rank.
\end{theorem}

Examples show that Theorem~\ref{second} does not hold without the assumption of $p$-solubility. We produce such examples for every prime $p$ in \S\,\ref{counterexamples}.

On the other hand, we prove the following result for arbitrary finite groups with a global rank-type conclusion under a stronger local condition.

\begin{theorem}\label{third}
Let $p$ be a prime, $r$ a positive integer, $G$ a finite group, and $P$ a Sylow $p$-subgroup of $G$. Suppose that
\begin{enumerate}
\item[\rm (a)] any subgroup generated by a subset of $I_G(P)$ can be generated by $r$ elements; and
\item[\rm (b)] for any $x\in I_G(P)$, any subgroup generated by a subset of $I_G(x)$ can be generated by $r$ elements.
\end{enumerate}
Then $[G,P]$ has $r$-bounded rank.
\end{theorem}

Another local--global result for arbitrary finite groups concerns the rank of the derived subgroup.

\begin{theorem}\label{fourth}
Let $G$ be a finite group. Suppose that for every prime $p$ dividing $|G|$, for any Sylow $p$-subgroup $P$, any subgroup generated by a subset of $I_G(P)$ can be generated by $r$ elements. Then the derived subgroup $G'$
has $r$-bounded rank.
\end{theorem}

The proofs of Theorems~\ref{second}, \ref{third}, \ref{fourth} depend on the classification of finite simple groups.

As an important tool in the proofs, we prove the following result about automorphisms, which is also of independent interest. When $A$ is a group acting by automorphisms on a group $G$, the subset $I_G(A)$ and the subgroup $[G,A]$ have the same meaning as above, in the natural semidirect product $GA$. We say that $\alpha$ is a coprime automorphism of a finite group $G$ if the order of $\alpha$ is coprime to the order of $G$, that is, $(|G|,|\alpha |)=1$.

\begin{theorem} \label{main}
Suppose that $G$ is a finite group admitting a group of coprime automorphisms $A$ such that, for a positive integer $r$, any subgroup generated by a subset of $I_G(A)$ can be generated by $r$ elements. Then $[G,A]$ has $r$-bounded rank.
\end{theorem}

The proof of Theorem~\ref{main} depends on the classification of
finite simple groups.

The set $I_G(A)$ is in a sense dual to the set of fixed points $C_G(A)$. For example, for an automorphism $\alpha$ we have $|I_G(\alpha)|=|G:C_G(\alpha)|$. There is a great deal of important results deriving nice properties of a finite group from various smallness conditions on the centralizers of automorphisms. Many of these results stem from the seminal papers of J.~G.~Thompson \cite{tho, tho64} and G.~Higman \cite{hig} on automorphisms with few fixed points. For the `dual' direction, in a few recent papers \cite{tams,AGS23,as} conditions on the sets $I_G(A)$ were also shown to strongly influence the structure of a finite group. Theorem~\ref{main} continues the line of research initiated by C.~Acciarri, R.~Guralnick, and P.~Shumyatsky \cite{tams} who proved a similar result \cite[Theorem~1.1]{tams} for the case of cyclic group of automorphisms $A$ with a weaker bound for the rank of $[G,A]$ depending also on the order of $A$. Many ideas developed in \cite{tams} are used in the proof of Theorem~\ref{main}, and we improve some of the lemmas from \cite{tams} here for not necessarily cyclic $A$ making sure that the bounds depend only on $r$ and not on $|A|$. In the proof we also use \cite[Theorem~1.2]{tams} saying that, for cyclic $A=\langle\alpha\rangle$, if every commutator $[g,\alpha]$ has odd order, then the subgroup $[G,A]$ has odd order.

 In the proofs we use the theory of powerful $p$-groups developed by A.~Lubotzky and A.~Mann \cite{LM}.
An important role is also played by results in representation theory. These results include Hall--Higman--type theorems, both from the original Hall--Higman paper \cite{ha-hi} with B.~Hartley's extensions \cite{hart}, and one of the so-called `non-modular' versions recently obtained by E.~I.~Khukhro and W.~Moens in \cite{khu-moe} on the basis of the well-known technique developed in the 1960s by E.~C.~Dade, E.~Shult \cite{shult}, F.~Gross \cite{gross}. Another result in representation theory used here is Theorem~B in the paper of B.~Hartley and I.~M.~Isaacs~\cite{HI}.
\bigskip\bigskip

\noindent {\sc Acknowledgements.} The authors thank the anonymous referee for careful reading and valuable comments.
\bigskip\bigskip

\section{Preliminaries}

All groups considered in this paper are finite. We use without explicit references the Feit--Thompson theorem saying that groups of odd order are soluble \cite{FT}. The following elementary lemma belongs to folklore.

\begin{lemma}\label{l-ker} Let $A$ be a group acting by automorphisms on a group $G$.
\begin{itemize}
 \item[\rm (a)] If $N$ is a normal subgroup of $G$ such that $N\leqslant C_G(A)$, then $[G,A]$ centralizes $N$.

 \item[\rm (b)] If $H$ is a subgroup such that $[H,H]=H$ and $[H,A]\leqslant C_G(H)$, then $[H,A]=1$.
\end{itemize}
\end{lemma}

\begin{proof}
(a) The hypothesis means that $A$ is contained in the kernel $K$ of the action of the semidirect product $GA$ by conjugation on $N$. Since $K$ is a normal subgroup of $GA$, we also have $[G,A]\leqslant K$.

(b) We have $[[H,A],H]\leqslant [C_G(H),H]=1$ and $[[A,H],H]\leqslant [C_G(H),H]=1$; hence $[[H,H],A]=[H,A]=1$ by the Three Subgroups Lemma.
\end{proof}

 The following lemma collects some well-known facts about coprime automorphisms of finite groups (see, for example, \cite{gore}); we shall sometimes use these facts without special references. We denote by $\pi(G)$ the set of primes dividing the order of a group $G$.

\begin{lemma}\label{l-20} Let $A$ be a group of coprime automorphisms of a group $G$.
\begin{itemize}

\item[\rm (a)] The group $G$ has an $A$-invariant Sylow $p$-subgroup for each prime $p\in\pi(G)$.

\item[\rm (b)] If $N$ is an $A$-invariant normal subgroup of $G$, then $C_{G/N}(A)=C_G(A)N/N$; in particular, $G=[G,A]C_G(A)$.

\item[\rm (c)] If $N$ is a normal $A$-invariant subgroup of $G$ such that $C_G(N)\leqslant N$, then $C_A(N)=1$.
\end{itemize}
\end{lemma}

 For a group $A$ acting by automorphisms on a group $G$, recall the definition of the subset $I_G(A)=\{[g,\alpha]\mid \alpha\in A\}$, where the commutators $[g,\alpha] =g^{-1}g^{\alpha}$ are considered in the natural semidirect product $GA$. Clearly, $[G,A]=\langle I_G(A)\rangle$. We shall use the following elementary properties without special references.

 \begin{lemma}\label{l-IGA}
 Let $A$ be a group acting by automorphisms on a group $G$.
\begin{itemize}

\item[\rm (a)] If $N$ is an $A$-invariant normal subgroup of $G$, then $I_{G/N}(A)=\{gN \mid g\in I_G(A)\}$.

\item[\rm (b)] If $A=\langle \alpha\rangle$ is cyclic, then $|I_G(\alpha)|=[G:C_G(\alpha)]$.
 \end{itemize}
 \end{lemma}

We also recall another well-known fact.

 \begin{lemma}\label{l-nil-rank}
 If $N$ is a nilpotent group of class $c$ generated by $k$ elements, then the rank of $N$ is bounded in terms of $k$ and $c$.
 \end{lemma}

The following lemma appeared independently and simultaneously in
the papers of Yu.~M.~Gorchakov~\cite{grc}, Yu.~I.~Merzlyakov~\cite{me}, and as ``P.~Hall's lemma" in the paper of J.~Roseblade~\cite{rs}.

\begin{lemma}\label{l-gmh}
 Let $p$ be a~prime number. The rank of a~$p$-group of
automorphisms of an abelian
finite $p$-group of rank~$r$ is bounded in
terms of~$r$.
\end{lemma}

We shall be using the following well-known result of A.~Lubotzky and A.~Mann \cite[Propositions~2.6 and 4.2.6]{LM}.

\begin{lemma}\label{l-24}
Let $p$ be a prime, and $G$ a group of exponent $p^k$ and of rank $m$. Then there is a number $s(k,m)$ depending only on $k$ and $m$ such that $|G|\leqslant p^{s(k,m)}$.
\end{lemma}

For a prime $p$, the largest normal $p$-subgroup of a group $G$ is denoted by $O_p(G)$, and the largest normal $p'$-subgroup by $O_{p'}(G)$. Then $O_{p',p}(G)$ denotes the inverse image of $O_p(G/O_{p'}(G))$, then $O_{p',p,p'}(G)$ is the inverse image of $O_{p'}(G/O_{p',p}(G))$, and so on; these subgroups form the so-called lower $p$-soluble series of $G$. A group $G$ is said to be $p$-soluble, if this series terminates with $O_{p',p,p'\dots}(G)=G$; then the number of symbols $p$ in the subscript is called the $p$-length of $G$. We shall need the following lemma.

\begin{lemma}\label{l-rpl}
Suppose that $G$ is a $p$-soluble group, $P$ is a Sylow $p$-subgroup of $G$, and $N$ is a normal subgroup of $P$ of rank $r$. Then the normal closure of $N$ in $G$ has $r$-bounded $p$-length.
\end{lemma}

\begin{proof}
 In other words, we need to show that $N$ is contained in $O_{p',p,p',p,...}(G)$, where the number of symbols $p$ in the subscript is $r$-bounded.

Consider the image $\bar N$ of $N$ in $G/O_{p',p}(G)$ in its action by conjugation on the Frattini quotient $V$ of $O_{p',p}(G)/O_{p'}(G)$ regarded as an $\mathbb{F}_p(G/O_{p',p}(G))$-module. By the Hall--Higman Theorem~B \cite[Theorem~B]{ha-hi}, if $g\in \bar N$ has order $p^k$, then its minimal polynomial on $V$ has degree at least $p^k-p^{k-1}$. This degree is known to be equal to the (maximum) dimension of the span of an orbit of a vector $v\in V$ under the action of $\langle g\rangle$; see, for example, \cite[Lemma~2.6(d)]{khu-moe}. This span is generated by $v$, $[v,g]$, $[v,g,g]$, .... All these elements, except~$v$, are images of elements of $N$, and therefore the dimension of the span is at most $r+1$. Thus, $p^k-p^{k-1}\leqslant r+1$. Hence $p^k$ is $r$-bounded, which means that the exponent of $\bar N$ is $r$-bounded.

Since the rank of $\bar N$ is at most $r$, it follows by Lemma~\ref{l-24} that the order $|\bar N|$ is $r$-bounded. In particular, the derived length of $\bar N$ is $r$-bounded. Then the Hall--Higman theorem \cite[Theorem 3.2.1]{ha-hi} for $p>3$ and B.~Hartley's results \cite[Theorem~2]{hart} for $p=2,3$ imply that $\bar N\leqslant O_{p',p,p',p,...}(G/O_{p',p}(G))$, where the number of symbols $p$ in the subscript is $r$-bounded. Hence the result.
\end{proof}

For brevity we say ``simple group'' meaning ``non-abelian simple group''.
We will often use without special references the well-known corollary of the classification that if a simple group $G$ admits a coprime automorphism $\alpha$ of order $e$, then $G=L(q)$ is a group of Lie type and $\alpha$ is a field automorphism. Furthermore, $C_G(\alpha)=L(q_0)$ is a group of the same Lie type defined over a smaller field such that $q=q_0^e$ (see \cite{GLS3}).

The following two lemmas were stated in \cite{tams} for cyclic $A$, but formally apply with any $A$, since any group of coprime automorphisms of a simple group is cyclic, being a subgroup of the group generated by the field automorphism induced by the Frobenius automorphism of the field.

\begin{lemma}[{\cite[Lemma~2.2]{tams}}]\label{l-0000} Let $r$ be a positive integer and $G$ a simple group admitting a group of coprime automorphisms $A$.
\begin{enumerate}
\item[\rm (a)] If the order of $[P,A]$ is at most $r$ whenever $P$ is an $A$-invariant Sylow subgroup of $G$, then the order of $G$ is $r$-bounded.

\item[\rm (b)] If the rank of $[P,A]$ is at most $r$ whenever $P$ is an $A$-invariant Sylow subgroup of $G$, then the rank of $G$ is $r$-bounded.
\end{enumerate}
\end{lemma}

\begin{lemma}[{\cite[Lemma~2.3]{tams}}]\label{l-1111} Let $G$ be a simple group admitting a group of coprime automorphisms $A$. There is a prime $p\in\pi(G)$ such that $G$ is generated by two $p$-subgroups $P_1$ and $P_2$ with the property that $P_1=[P_1,A]$ and $P_2=[P_2,A]$.
\end{lemma}

We call a direct product of simple groups a \emph{semisimple group}. The next lemma is proved similarly to \cite[Lemma~2.4]{tams}.

\begin{lemma}\label{01} Let $C$ be a positive integer and $G$ a finite group admitting a group of coprime automorphisms $A$ such that $G=[G,A]$. Suppose that the order of $[P,A]$ is at most $C$ whenever $P$ is an $A$-invariant Sylow subgroup of $G$. Then the order of $G$ is $C$-bounded.
\end{lemma}

\begin{proof} If $G$ is abelian, then for every prime $p$ the unique Sylow $p$-subgroup satisfies $P=[P,A]$ and $p\leqslant |P|\leqslant C$. It follows that $|G|\leqslant C^f$, where $f$ is the number of primes less than or equal to $C$. So we assume that $G$ is nonabelian.

First suppose that $G$ has no proper $A$-invariant normal subgroups. Being non-abelian, then $G=G_1\times\dots \times G_l$ is a direct product of isomorphic simple groups transitively permuted by $A$. If $l=1$, then the result immediately follows from Lemma~\ref{l-0000}(a). If $l>1$, then for any prime $q\in \pi (G)$ an $A$-invariant Sylow $q$-subgroup $Q$ is a product $Q_1\times\dots\times Q_l$ of Sylow subgroups of the $G_i$ and $A$ transitively permutes the factors $Q_i$. We observe that $C\geqslant |[Q,A]|\geqslant|Q_1|^{l-1}$ and the result follows since any prime dividing $|G|$ is at most $C$, while $|G|= \prod_{q\in \pi (G)}|Q_1|^l$, with $|Q_1|^l\leqslant |Q_1|^{2l-2}\leqslant C^2$.

 So suppose that $G$ has proper $A$-invariant normal subgroups. Let $\pi(G)=\{p_1,\dots,p_k\}$ and for each $i\leqslant k$ choose an $A$-invariant Sylow $p_i$-subgroup $P_i$ of $G$. We set $s(G)=\prod_{1\leqslant i\leqslant k}|[P_i,A]|$. Note that the number of factors greater than~1 is $C$-bounded, since for such factors $p_i\leqslant |[P_i,A]|\leqslant C$. Hence, $s(G)$ is $C$-bounded and we can use induction on $s(G)$. Suppose first that $A$ acts nontrivially on every $A$-invariant normal subgroup of $G$. Let $M$ be a minimal $A$-invariant normal subgroup. Since $[M,A]\ne 1$, we have $s(G/M)<s(G)$. By induction the order of $G/M$ is $C$-bounded. The subgroup $M$ is either an elementary abelian $p$-group for some prime $p\leqslant C$ or a semisimple group. In any case, $[M,A]$ has $C$-bounded order. Since $[M,A]$ is normal in $M$ and $M$ has $C$-bounded index in $G$, the normal closure $\langle[M,A]^G\rangle$ has $C$-bounded order. By the minimality of $M$ we have $\langle[M,A]^G\rangle=M$, and so the order of $G$ is $C$-bounded. This completes the proof in the particular case where $A$ acts nontrivially on every $A$-invariant normal subgroup of $G$.

Next, suppose that $G$ has nontrivial normal subgroups contained in $C_G(A)$. Let $N$ be the product of all such subgroups. In view of the above, $G/N$ has $C$-bounded order. Since $N\leqslant Z(G)$ by Lemma~\ref{l-ker}(a), we deduce from Schur's theorem \cite[Theorem 4.12]{rob} that $G'$ has $C$-bounded order. Hence the result.
\end{proof}

In the proof of the next lemma, which is similar to \cite[Lemma~2.5]{tams}, we shall use the following theorem of
B.~Hartley and I.~M.~Isaacs \cite{HI}.

\begin{theorem}[{\cite[Theorem~B]{HI}}]\label{t-HI}
Let $A$ be an arbitrary finite group. Then there exists a
number $\varepsilon=\varepsilon(A)> 0$, depending only on $A$, with the following property.
Let $A$ act by coprime automorphisms on a soluble finite group $H$, and let $k$ be any
field with characteristic not dividing $|A|$. Let $V$ be any simple $kHA$-module
and let $S$ be any simple $kA$-module which appears as a component of the restriction $V|_A$. Then the multiplicity of $S$ in $V|_A$ is at least $\varepsilon \dim V$.
\end{theorem}

\begin{lemma}\label{l-02} Let $G=HA$ be a group with a normal subgroup $H$ and a subgroup $A$ such that $(|H|,|A|)=1$ and $H=[H,A]$. Suppose that $G$ faithfully acts by permutations on a set $\Omega$ in such a way that the subgroup $A$ moves only $m$ points. Then the order of $G$ is $m$-bounded.
\end{lemma}

\begin{proof} First, note that the order of $A$ is obviously $m$-bounded.
Another useful observation is that because of Lemma \ref{01} we can assume without loss of generality that $H$ is a $p$-group for some prime $p$.

Let $\Omega_0$ be a nontrivial $G$-orbit. If $A$ moves no points in $\Omega_0$, then $H=[H,A]$ also acts trivially on $\Omega_0$, a contradiction. Therefore, $A$ moves at least $2$ points on every nontrivial $G$-orbit and so there are at most $m/2$ nontrivial orbits of $G$ in $\Omega$. Since $G$ embeds into the direct product $G_1\times \cdots \times G_k$, where $G_i$ is the image of $G$ in its action on the $i$th nontrivial $G$-orbit, and $k\leqslant m/2$, it is sufficient to obtain a bound in terms of $m$ for each $|G_i|$. Each $G_i$ satisfies the hypotheses of the lemma with $G_i=H_iA_i$ in place of $G=HA$, where $H_i$ and $A_i$ are the images of $H$ and $A$. Thus, we can assume without loss of generality that $G$ acts transitively on $\Omega$ and hence it is sufficient to bound the cardinality of $\Omega$ in terms of $m$.

Consider the corresponding permutational representation of $G$ over $\mathbb{C}$, where $G$ naturally acts on the $|\Omega|$-dimensional linear space $V$. The dimension of $[V,A]$ is at most $m-1$. The space $V$ is a direct sum of irreducible $\mathbb{C}G$-modules and there are at most $m$ nontrivial irreducible $\mathbb{C}A$-submodules. If $A$ acts trivially on a $\mathbb{C}G$-module, then $H=[H,A]$ also acts trivially, so this $\mathbb{C}G$-module is the 1-dimensional fixed-point subspace of $G$.  On every other irreducible $G$-module, $A$ acts nontrivially, and hence this $G$-module has $(m,|A|)$-bounded dimension by the Hartley--Isaacs Theorem~\ref{t-HI}, which is applicable since $H$ is a $p$-group. Since $|A|$ is $m$-bounded, it follows that the dimension of $V$ is $m$-bounded, as required.
\end{proof}

\begin{lemma}\label{l-autsem}
If $A$ is a group of coprime automorphisms of a semisimple group $H$ of rank~$r$, then $A$ has $r$-bounded order.
\end{lemma}

\begin{proof}
 Let $H=S_1\times\dots\times S_l$ be a direct product of simple factors $S_i$. Since every $S_i$ contains an involution, we clearly have $l\leqslant r$. Since $A$ permutes these factors, there is a subgroup $A_1$ of $r$-bounded index in $A$ that stabilizes each factor. So for every factor, $A_1$ induces a group of coprime automorphisms. Therefore this induced group is non-trivial only if this factor is a group of Lie type $L(q^e)$, where $q$ is a prime, and the induced group is a subgroup $\langle\psi\rangle$ of the cyclic group of field automorphisms of order dividing $e$. Since $L(q^e)$ contains an elementary abelian $q$-group of rank $e$, the bound for the rank implies that $e\leqslant r$, and so the order of $\psi$ is $r$-bounded. Hence the order of $A$ is $r$-bounded.
\end{proof}

\section{Proof of Theorem~\ref{main}}

Suppose that $G$ is a finite group admitting a group of coprime automorphisms $A$ such that any subgroup generated by a subset of $I_G(A)$ can be generated by $r$ elements. Our aim is to prove that $[G,A]$ has $r$-bounded rank.
The overall strategy involves the following steps. First we prove the result for the case of nilpotent group~$G$. Then we consider semisimple normal subgroups in the general case. Next, we use a version of the `non-modular' Hall--Higman type theorem to obtain a bound in terms of $r$ for the exponent of the automorphisms induced by $A$ on the quotient by the generalized Fitting subgroup $G/F^*(G)$. Combined with the previous result in \cite{tams}, this bound gives the proof of Theorem~\ref{main} for the case of cyclic group $A$, as well as a bound in terms of $r$ for the Fitting height of soluble $A$-invariant sections $S$ satisfying $S=[S,A]$. The proof is completed by reduction to the case of cyclic~$A$.

\subsection{The case of nilpotent groups}

First we settle the case where $G$ is nilpotent, by largely following the arguments in~\cite{tams}. Recall the usual notation $Z_i(H)$ and $\gamma_i(H)$ for the $i$-th term of the upper and lower central series of a group $H$, respectively.

\begin{lemma}\label{l-22}
Let $G$ be a group admitting a group of coprime automorphisms $A$ such that $G=[G,A]$. Let $p$ be a prime and suppose that $M$ is an $A$-invariant normal $p$-subgroup of $G$ such that $|[M,A]|=p^m$ for some non-negative integer $m$. Then $M\leqslant Z_{2m+1}(O_p(G))$.
\end{lemma}

\begin{proof} We use induction on $m$. If $m=0$, then $M\leqslant C_G(A)$ and therefore $M\leqslant Z(G)$ by Lemma~\ref{l-ker}(a). Now let $m\geqslant 1$. Set $K=O_p(G)$ to lighten the notation.
 If $M\leqslant Z(K)$, there is nothing to prove. If $M\not\leqslant Z(K)$, then the image of $M$ in $K/Z(K)$ has a non-trivial intersection with the centre of this quotient. In other words, $M\cap Z_2(K) \not\leqslant Z(K)$. Then Lemma~\ref{l-ker}(a) implies that $[M\cap Z_2(K), A]\neq 1$. It follows that $|[M/(M\cap Z_2(K)), A]|<|[M,A]|=p^m$. Indeed,
 $$
 [M/(M\cap Z_2(K)), A]=[M,A](M\cap Z_2(K))/(M\cap Z_2(K))\cong [M,A]/([M,A]\cap (M\cap Z_2(K))),
 $$
 where $[M,A]\cap (M\cap Z_2(K))\geqslant [M\cap Z_2(K), A]\ne 1$. Thus, $|[M/(M\cap Z_2(K)), A]|\leqslant p^{m-1}$.
 By induction, $M/(M\cap Z_2(K)) \leqslant Z_{2m-1}(K/(M\cap Z_2(K)))$, whence $M\leqslant Z_{2m+1}(K)$, as required.
\end{proof}

Throughout the rest of this subsection, unless stated otherwise, $G$ is a $p$-group admitting a group of coprime automorphisms $A$ such that $G=[G,A]$ and any subgroup generated by a subset of $I_G(A)$ can be generated by $r$ elements.

\begin{lemma}\label{l-23} Suppose that $G$ is of prime exponent $p$ or of exponent~4. There exists a number $l(r)$ depending on $r$ only such that the rank of $G$ is at most $l(r)$.
\end{lemma}

\begin{proof}
Let $C$ be Thompson's critical subgroup of $G$ (see \cite[Theorem 5.3.11]{gore}).
Observe that $[Z(C),A]$ is an $r$-generated abelian subgroup of exponent $p$ (or 4) and so the order of $[Z(C),A]$ is at most $p^r$ (or $2^{2r})$. By Lemma \ref{l-22} $Z(C)$ is contained in $Z_{2r+1}(G)$ (or in $Z_{4r+1}(G)$). Since $[G,C]$ is contained in $Z(C)$, we conclude that $C$ is contained in $Z_{2r+2}(G)$ (or in $Z_{4r+2}(G)$). Recall that $\gamma_{2r+2}(G)$ commutes with $Z_{2r+2}(G)$ and so in particular $\gamma_{2r+2} (G)$ (respectively, $\gamma_{4r+2} (G)$) centralizes $C$. By Thompson's theorem, $C_G(C)=Z(C)$. Thus $\gamma_{2r+2}(G)$ (respectively, $\gamma_{4r+2}(G)$) is contained in $Z(C)$, that is, the quotient $G/Z(C)$ is nilpotent of class $2r+1$ (respectively, of class $4r+1$). Since $Z(C)\leqslant Z_{2r+1}(G)$ (or $Z(C)\leqslant Z_{4r+1}(G)$), it follows that $G$ has $r$-bounded nilpotency class. Since $G=[G,A]$ is $r$-generated by hypothesis, by Lemma~\ref{l-nil-rank} the rank of $G$ is $r$-bounded, as desired.
\end{proof}

We will require the concept of powerful $p$-groups introduced by A.~Lubotzky and A.~Mann in \cite{LM}. A finite $p$-group $H$ is \emph{powerful} if and only if $[H,H] \leqslant H^p$ for $p\neq 2$ (or $[H,H]\leqslant H^4$ for $p = 2$).
Apart from the original paper \cite{LM}, information about the properties of powerful $p$-groups can also be found in the books \cite{DDM} or \cite{khukhu2}.

\begin{lemma}\label{l-powerful}
There exists a number $\lambda=\lambda(r)$ depending only on $r$ such that $\gamma_{2\lambda+1}(G)$ is powerful.
\end{lemma}

\begin{proof}
Let $s'(m)=s(1,m)$ if $p\ne 2$, and $s'(m)=s(2,m)$ if $p=2$ for the function $s(k,m)$ as in Lemma~\ref{l-24}, and let $l(r)$ be as in Lemma~\ref{l-23}. Take $N=\gamma_{2\lambda+1}(G)$, where $\lambda=s'(l(r))$. In order to show that $N'\leqslant N^p$ (or $N'\leqslant N^4$ when $p=2$), we assume that $N$ is of exponent $p$ (or 4) and prove that $N$ is abelian.

Since the subgroup $[N,A]$ is of exponent $p$ (or $4$), the rank of $[N,A]$ is at most $l(r)$ by Lemma~\ref{l-23}. Then $|[N,A]|\leqslant p^{s'(l(r))}=p^\lambda$ by Lemma~\ref{l-24}. We now obtain $N\leqslant Z_{2\lambda+1}(G)$ by Lemma~\ref{l-22}. Since $[\gamma_i(G),Z_i(G)]=1$ for any positive integer $i$, we conclude that $N$ is abelian, as required.
\end{proof}

\begin{lemma}\label{l-25}
For any $i\geqslant 1$, there exists a number $m_i=m(i,r)$ depending only on $i$ and $r$ such that $\gamma_i(G)$ is an $m_i$-generated group.
\end{lemma}

\begin{proof}
Let $N=\gamma_i(G)$. We can pass to the quotient $G/\Phi(N)$ and assume that $N$ is elementary abelian. It follows that $|[N,A]|\leqslant p^r$. Then $N\leqslant Z_{2r+1}(G)$ by Lemma \ref{l-22}, and therefore $G$ has nilpotency class bounded only in terms of $i$ and $r$. Since $G=[G,A]$ is $r$-generated, the rank of $G$ is also $(i,r)$-bounded by Lemma~\ref{l-nil-rank}. In particular, $N$ is $m_i$-generated for some $(i,r)$-bounded number $m_i$.
\end{proof}

The next proposition shows that Theorem \ref{main} is valid in the case where $G$ is a $p$-group.

\begin{proposition}\label{p-pcase}
Suppose that $G$ is a $p$-group admitting a group of coprime automorphisms $A$ such that any subgroup generated by a subset of $I_G(A)$ can be generated by $r$ elements. Then $[G,A]$ has $r$-bounded rank.
\end{proposition}

\begin{proof}
We can obviously assume that $G=[G,A]$. Let $s'(m)=s(1,m)$ if $p\ne 2$, and $s'(m)=s(2,m)$ if $p=2$ for the function $s(k,m)$ as in Lemma~\ref{l-24}, and let $l(r)$ be as in Lemma \ref{l-23}. Take $N=\gamma_{2\lambda+1}(G)$, where $\lambda=\lambda(r)=s'(l(r))$. Let $d$ be the minimum number of generators of $N$. The number $d$ is $r$-bounded by Lemma~\ref{l-25}, and $N$ is a powerful $p$-group by Lemma \ref{l-powerful}. By the properties of powerful $p$-groups (see, for example, \cite[Theorem 2.9]{DDM}) the rank of $N$ is equal to $d$ and therefore is $r$-bounded. Since the nilpotency class of $G/N$ is $r$-bounded (recall that $\lambda$ depends only on $r$) and $G=[G,A]$ is $r$-generated, the rank of $G/N$ is also $r$-bounded by Lemma~\ref{l-nil-rank}. Since the rank of $G$ is at most the sum of the ranks of $G/N$ and $N$, the result follows.
\end{proof}

\begin{corollary}\label{c-nilp} Suppose that $G$ is a nilpotent group admitting a group of coprime automorphisms $A$ such that any subgroup generated by a subset of $I_G(A)$ can be generated by $r$ elements. Then $[G,A]$ has $r$-bounded rank.
\end{corollary}

\begin{proof} The rank of $[G,A]$ is equal to the rank of $[P,A]$, where $P$ is some Sylow $p$-subgroup of $G$, and the result follows from Proposition~\ref{p-pcase}.
\end{proof}

\subsection{The case of semisimple groups}

Let $G$ be a finite group admitting a group of coprime automorphisms $A$
such that any subset of $I_G(A)$ generates an $r$-generator subgroup. We want to prove that $[G,A]$ has $r$-bounded rank. Throughout this subsection we
assume that $G=[G,A]$.

\begin{lemma}\label{l-1}
If $G$ is simple, then the rank of $G$ is $r$-bounded.
\end{lemma}

\begin{proof} For any $A$-invariant Sylow subgroup $P$ of $G$ the rank of $[P,A]$ is $r$-bounded by Proposition~\ref{p-pcase}. Then $G$ has $r$-bounded rank by
Lemma~\ref{l-0000}(b).
\end{proof}

Recall that a semisimple group is a direct product of non-abelian simple groups.

\begin{lemma}\label{l-2} Suppose that $G$ is semisimple and $A$ transitively permutes the simple factors. Then the rank of $G$ is $r$-bounded.
\end{lemma}

\begin{proof}
Write $G=S_1\times\dots\times S_k$, where the $S_i$ are simple groups. Since the case $k=1$ was considered in Lemma \ref{l-1}, we assume that $k\geqslant2$.

First we note that $k$ is at most $r+1$. Indeed, say, for an involution $t\in S_1$ there are elements $\alpha_i\in A$ such that $t^{\alpha_i}\in S_i$ for all $i$. Then the rank of the subgroup $\langle [t,\alpha_i]\mid 1\leqslant i\leqslant k\rangle$ is at least $k-1$, so that $k-1\leqslant r$ by hypothesis.

Hence it is sufficient to show that the rank of $S_1$ is at most $r$, that is, every subgroup $H$ of $S_1$ can be generated by $r$ elements. Using the assumption $k\geqslant2$, we pick $\alpha\in A$ such that $S_1^{\alpha}=S_2$. Consider the subgroup $K\leqslant S_1\times S_2$ generated by all elements of the form $x^{-1}x^\alpha$, where $x\in H$. Since $K$ is generated by a subset of $I_G(A)$, it can be generated by $r$ elements. Hence $H$, which is the projection of $K$ onto $S_1$, can also be generated by $r$ elements.
\end{proof}

\begin{lemma}\label{l-3} Suppose that $G$ is semisimple. Then the rank of $G$ is $r$-bounded.
\end{lemma}

\begin{proof}
Since $G=[G,A]$, it follows that $G=G_1\times\dots\times G_m$, where each factor $G_i$ is either simple such that $G_i=[G_i,A]$ or a direct product, with at least two factors, of simple groups transitively permuted by $A$. By the two previous Lemmas~\ref{l-1} and \ref{l-2} the rank of each $G_i$ is $r$-bounded. Hence it remains to show that the number $m$ of such factors is also $r$-bounded. For every $i$ there is $\alpha_i\in A$ such that $[G_i, \alpha_i]\ne 1$. By \cite[Theorem~1.2]{tams} each subgroup $G_i$ has an element $g_i$ such that $[g_i, \alpha_i]$ has even order. The abelian subgroup $\langle [g_1, \alpha_1],\dots,[g_m, \alpha_m]\rangle$ has a Sylow 2-subgroup of rank $m$ and therefore it cannot be generated by less than $m$ elements. This subgroup is generated by a subset of $I_G(A)$, and therefore $m\leqslant r$.
\end{proof}

\begin{lemma}\label{l-444} Let $N$ be an $A$-invariant normal subgroup of $G$ which is a direct product $N=S_1\times\dots\times S_l$ of simple factors $S_i$. Then both $l$ and the rank of $N$ are $r$-bounded.
\end{lemma}

\begin{proof} By Lemma \ref{l-3} the rank of $[N,A]$ is $r$-bounded. Since all factors $S_i$ have even order and the rank of the Sylow $2$-subgroup of $[N,A]$ is $r$-bounded, it follows that only $r$-boundedly many, say $m$, of the subgroups $S_1,\dots,S_l$ are not contained in $C_G(A)$. On the other hand, since $G=[G,A]$, it follows by Lemma~\ref{l-ker}(a) that $C_N(A)$ cannot contain a nontrivial normal subgroup of $G$. Hence every simple factor among $S_1,\dots,S_l$ is conjugate in $G$ with a factor which is not centralized by $A$ and so each $S_i$ has $r$-bounded rank by Lemma \ref{l-3}. Thus, we only need to show that $l$ is $r$-bounded.

The group $GA$ naturally acts on the set $\Omega=\{S_1,\dots,S_l\}$ by conjugation. The above argument shows that the number of $GA$-orbits in this action is at most $m$. It is sufficient to show that each $GA$-orbit has $r$-bounded length. Let $K$ be the kernel of the action, that is, $K=\bigcap_i N_{GA}(S_i)$. By Lemma \ref{l-02} the order $|GA/K|$ is $m$-bounded. Since the length of each $GA$-orbit is at most $|GA/K|$, the result follows, since $m$ is $r$-bounded.
\end{proof}

\subsection{\boldmath The exponent of the automorphisms induced by $A$ on $G/F^*(G)$}

Let $G$ be a finite group admitting a group of coprime automorphisms $A$
such that any subset of $I_G(A)$ generates an $r$-generator subgroup. Recall that these hypotheses are inherited by $I_S(B)$ for any $B$-invariant section $S$ of $G$ for any subgroup $B\leqslant A$. The aim is to prove that $[G,A]$ has $r$-bounded rank. In this subsection we focus on the action of $A$ on the quotient by the generalized Fitting subgroup $F^*(G)$. Recall that $F^*(G)=F(G)E(G)$, where $F(G)$ is the Fitting subgroup, while $E(G)=Q_1*\dots *Q_k$, known as the layer, is the central product of all subnormal quasisimple subgroups $Q_i$, that is, perfect groups with non-abelian simple central quotients. Our aim in this subsection is a bound in terms of $r$ for the exponent of the automorphisms induced by $A$ on $G/F^*(G)$. As a consequence, we shall finish the proof of Theorem~\ref{main} for the case of cyclic $A$, which will be also used later.

\begin{lemma}\label{l-k1}
The rank of $[F^*(G),A]$ is $r$-bounded.
\end{lemma}

\begin{proof}
 The rank of $[F(G),A ]$ is $r$-bounded by Corollary~\ref{c-nilp}.
 The rank of $[E(G),A]/Z([E(G),A ])$ is $r$-bounded by Lemma~\ref{l-3}. The rank of $Z([E(G),A ])$ is also $r$-bounded. Indeed, this is a product of the centres of the quasisimple factors involved in $[E(G),A ]$, because if $[Q_i,A ]\leqslant Z(E(G))$ for a quasisimple factor $Q_i$, then $[Q_i,A ]=1$ by Lemma~\ref{l-ker}(b). The number of quasisimple factors involved in $[E(G),A ]$ is $r$-bounded by Lemma~\ref{l-444}, and each has centre of rank at most $2$, as known from the classification. Hence the result.
\end{proof}

Let $\alpha \in A$ be any element of $A$.

\begin{lemma}\label{l-k2}
The automorphism of $[E(G),\alpha]$ induced by $\alpha$ has $r$-bounded order.
\end{lemma}

\begin{proof}
The automorphism $\alpha$ permutes the quasisimple factors $Q_i$ in $E(G)=Q_1*\dots *Q_k$. As in the proof of Lemma~\ref{l-2}, the orbits have length at most $r+1$. If the stabilizer of a point $Q_i$ in $\langle \alpha\rangle$ does not centralize $Q_i$, then it induces a non-trivial coprime automorphism $\beta$ of the simple group $Q_i/Z(Q_i)$. Then this $Q_i/Z(Q_i)$ must be a group of Lie type, say, $L(q^f)$ over a field $\mathbb{F}_{q^f}$ for a prime $q$, and $\beta$ must be a field automorphism of order dividing~$f$. The rank of this $L(q^f)$ is $r$-bounded by Lemma~\ref{l-k1}. In particular, $f$ must be $r$-bounded, since $L(q^f)$ contains an elementary abelian $q$-subgroup of order $q^f$. We have proved that the automorphism induced by $\alpha$ on $E(G)/Z(E(G))$ has $r$-bounded order. This is also the order of the automorphism induced by $\alpha$ on $E(G)$, since if $[E(G),\alpha ^s]\leqslant Z(E(G))$, then $[E(G),\alpha ^s]=1$ by Lemma~\ref{l-ker}(b). Hence the result.
\end{proof}

\begin{proposition}\label{p-exp}
Let $G$ be a finite group admitting a group of coprime automorphisms $A$
such that any subset of $I_G(A)$ generates an $r$-generator subgroup. There is an $r$-bounded positive integer $n=n(r)$ such that $A^n$ acts trivially on $G/F^*(G)$. \end{proposition}

\begin{proof}
 In other words, we need to show that for any $\alpha \in A$ the order of the automorphism $\bar \alpha$ induced by $\alpha$ on $G/F^*(G)$ is $r$-bounded.
It is sufficient to bound in terms of $r$ the order of every Sylow $p$-subgroup of $\langle \bar \alpha\rangle$, since then there will be only few such primes $p$, and $|\bar \alpha |$ will be bounded in terms of $r$. Therefore, replacing $\langle \alpha\rangle$ with its Sylow $p$-subgroup, we can simply assume that $|\bar \alpha |=p^m$ for a prime $p$; then we can also assume that $\alpha$ is a $p$-element.

The generalized Fitting subgroup of the quotient $G/\Phi (F(G))$ by the Frattini subgroup of $F(G)$ is the image of $F^*(G)$. Therefore we can assume that $\Phi (F(G))=1$, so that $F(G)$ is a direct product of elementary abelian groups.

 The automorphism of $E(G)$ induced by $\alpha $ has $r$-bounded order by Lemma~\ref{l-k2}, so if $\alpha ^{p^{m-1}}$ does not centralize $E(G)$, then $|\bar \alpha |=p^m$ is $r$-bounded. For the same reason, if $\alpha $ induces an automorphism of $E(G/F^*(G))$ of order at least $p^m$, then $|\bar \alpha |=p^m$ is $r$-bounded by Lemma~\ref{l-k2} applied to $G/F^*(G)$ and its group of automorphisms induced by $A$. Therefore we can assume that $\alpha ^{p^{m-1}}$ centralizes both $E(G)$ and $E(G/F^*(G))$.

 Since $F^*(G/F^*(G))$ contains its centralizer in $G/F^*(G)$ and $\bar \alpha$ is a coprime automorphism, $\bar \alpha$ acts faithfully on $F^*(G/F^*(G))$ by Lemma~\ref{l-20}(c). Since $[E(G/F^*(G)), \alpha ^{p^{m-1}}]=1$ by our assumption, $\bar \alpha $ acts faithfully on $F(G/F^*(G))$. Hence there is a prime $s$ such that $\langle \bar \alpha \rangle$ acts faithfully on the Sylow $s$-subgroup $S$ of $F(G/F^*(G))$, so that $[S, \alpha ^{p^{m-1}}]\ne 1$.

Let $\hat S$ be an $s$-subgroup that is a preimage of $S$ in $G$. Note that $C_{\hat S}(F(G)_{s'}E(G))\leqslant F(G)$, where $F(G)_{s'}$ is the Hall $s'$-subgroup of $F(G)$, because $C_{\hat S}(F(G)_{s'}E(G))F(G)$ is a nilpotent normal subgroup. It follows that $[\hat S, \alpha ^{p^{m-1}}]$ acts nontrivially on $F(G)_{s'}E(G)$.

 Since $[E(G), \alpha ^{p^{m-1}}]=1$ by our assumption, $[\hat S, \alpha ^{p^{m-1}}]$ centralizes $E(G)$ by Lemma~\ref{l-ker}(a). Therefore there is a prime $t\ne s$ and the Sylow $t$-subgroup $T$ of $F(G)$ on which $[\hat S, \alpha ^{p^{m-1}}]$ acts non-trivially; this is the same action as that of $[S, \alpha ^{p^{m-1}}]$.
 Let $\tilde S=\hat S/C_{\hat S}(T)$. Then the semidirect product $\tilde S \langle \alpha\rangle$ naturally acts on the elementary abelian $t$-group~$T$ in such a way that $[\tilde S, \alpha ^{p^{m-1}}]$ acts non-trivially, and $\tilde S$ acts faithfully. We can regard $T$ as a vector space over $\mathbb{F}_t$ and an $\mathbb{F}_t\tilde S \langle \alpha\rangle$-module.

 This is a kind of ``non-modular Hall--Higman type'' configuration for the group $\tilde S\langle \alpha\rangle$ acting on the vector space~$T$. This is not quite the usual setting as in the classical papers with non-modular Hall--Higman--type theorems, because $\langle \alpha\rangle $ may not be faithful on $\tilde S$, that is, $\alpha $ may have order greater than $p^m=|\bar \alpha |$. But this new situation was analysed in the recent paper of E.~I.~Khukhro and W.~Moens \cite{khu-moe}, with a proof similar to the classical case.

\begin{lemma}[{\cite[Lemma~3.3(a)]{khu-moe}}]\label{l-khu-moe}
 Suppose that $q$ is a prime, $Q$ is a $q$-group, and $\langle\psi\rangle$ is a cyclic group of order $p^k$ for a prime $p\ne q$ acting (not necessarily faithfully) by automorphisms on $Q$ such that the induced automorphism of $Q$ has order $p^m$. Suppose that the semidirect product $Q\langle\psi\rangle$ acts by linear transformations on a vector space $V$ over a field $K$ whose characteristic does not divide $p\cdot q$, and suppose that $V$ is a faithful $KQ$-module.
Then the minimal polynomial of $\psi$ on $V$ has degree at least $p^m-p^{m-1}$.
\end{lemma}

Applying this lemma with $V=T$, $Q=\tilde S$, and $\psi= \alpha$ we obtain that
the minimum polynomial of $\alpha $ on $T$ has degree at least $p^m-p^{m-1}$. This degree is known to be the maximum dimension of the span of an orbit under $\langle \alpha \rangle$ on $T$; see, for example, \cite[Lemma~2.6(d)]{khu-moe}. Hence the dimension of $[T,\alpha ]$, which is the same as the rank of this abelian subgroup, is at least $p^m-p^{m-1}-1$. By hypothesis, $p^m-p^{m-1}-1\leqslant r$. Hence $p^m$ is bounded in terms of~$r$.
\end{proof}

We combine Proposition~\ref{p-exp} with one of the main results of \cite{tams} to prove Theorem~\ref{main} for the case of cyclic $A$.

\begin{corollary}\label{c-cyclic}
Let $G$ be a finite group admitting a coprime automorphism $\alpha$
such that any subset of $I_G(\langle\alpha\rangle )$ generates an $r$-generator subgroup. Then the rank of $[G,\alpha ]$ is $r$-bounded.
\end{corollary}

\begin{proof}
By Proposition~\ref{p-exp} there is an $r$-bounded positive integer $n=n(r)$ such that $[G,\alpha^n]\leqslant F^*(G)$. Since $\alpha$ is a coprime automorphism, we
have $G=C_G(\alpha ^n)F^*(G)$, so that $[F^*(G),\alpha^n]=[G,\alpha ^n]$. Hence $[F^*(G),\alpha^n]$ is a normal $\alpha$-invariant subgroup. The order of the automorphism of $G/[F^*(G),\alpha^n]$ induced by $\alpha$ divides $n$ and so is $r$-bounded. By \cite[Theorem~1.1]{tams} the rank of $[G,\alpha]/[F^*(G),\alpha^n]$ is bounded in terms of $n$ an $r$ and therefore is $r$-bounded. The rank of $[F^*(G),\alpha^n]$ is $r$-bounded by Lemma~\ref{l-k1}. As a result, the rank of $[G,\alpha]$ is also $r$-bounded.
\end{proof}

\subsection{The case of soluble groups}

First we derive a useful corollary on the Fitting height of a soluble group satisfying the hypothesis of Theorem~\ref{main}.

\begin{corollary}\label{c-sol-f}
Suppose that $G$ is a finite soluble group admitting a group of coprime automorphisms $A$ such that any subgroup generated by a subset of $I_G(A)$ can be generated by $r$ elements. Then $[G,A]$ has $r$-bounded Fitting height.
\end{corollary}

\begin{proof}
Clearly, $[G,A]=\prod_{\alpha\in A}[G,\alpha ]$. Every subgroup $[G,\alpha ]$ has $r$-bounded rank by Corollary~\ref{c-cyclic}. The Fitting height of a finite soluble group is known to be bounded in terms of its rank; see for example, \cite[Lemma~2.4]{glasgow}. Therefore each of the normal subgroups $[G,\alpha ]$ has $r$-bounded Fitting height. Hence so does their product.
\end{proof}

We now prove Theorem~\ref{main} in the case of soluble groups.

\begin{proposition}\label{p-sol}
Suppose that $G$ is a finite soluble group admitting a group of coprime automorphisms $A$ such that, for a positive integer $r$, any subgroup generated by a subset of $I_G(A)$ can be generated by $r$ elements. Then $[G,A]$ has $r$-bounded rank.
\end{proposition}

\begin{proof}
We can obviously assume that $G=[G,A]$.
The Fitting height of $G$ is $r$-bounded by Corollary~\ref{c-sol-f}; therefore we can proceed by induction on the Fitting height of $G$. At the base of induction the group $G=[G,A]$ is nilpotent and then the rank of $G$ is $r$-bounded by Corollary~\ref{c-nilp}. Now suppose that the Fitting height of $G$ is greater than 1. The rank of $G/F(G)$ is $r$-bounded by the induction hypothesis, and it remains to show that the rank of $F(G)$ is $r$-bounded. The rank of $F(G)$ is equal to the rank of some Sylow $p$-subgroup of $F(G)$. This Sylow $p$-subgroup of $F(G)$ is isomorphic to a subgroup of $P=O_{p',p}(G)/O_{p'}(G)$, so it is sufficient to prove that the rank of $P$ is $r$-bounded.

Let $\bar G=G/O_{p'}(G)$. The subgroup $[P,A]$ has $r$-bounded rank by Corollary~\ref{c-nilp}. Hence the group $A/C_A([P,A])$ has $r$-bounded rank as a $p'$-group faithfully acting on the Frattini quotient of $[P,A]$ regarded as a vector space over $\mathbb{F}_p$ of $r$-bounded dimension (see, for example, \cite[Lemma~2.3]{glasgow}). Due to coprime action we have $C_A([P,A])=C_A(P)$. Then $[\bar G,C_A(P)]$ centralizes $P$ by Lemma~\ref{l-ker}(a). But $P$ contains its centralizer, so $[\bar G,C_A(P)]\leqslant P$, whence $[[\bar G,C_A(P)], C_A(P)]\leqslant [P,C_A(P]=1$ and therefore, $[\bar G,C_A(P)]=1$ due to coprime action. As a result, the group $\bar G$ admits the group of coprime automorphisms $\bar A=A/C_A(P)$, which has $r$-bounded rank and, in particular, is generated by $r$-bounded number of elements $a_1,\dots,a_{f(r)}$. By Corollary~\ref{c-cyclic} every subgroup $[\bar G, a_i]$ has $r$-bounded rank. Since
$$[\bar G,A]=\prod_{i=1}^{f(r)}[\bar G,a_i],$$
it follows that $\bar G=[\bar G,A]$ has $r$-bounded rank, so in particular, $P$ has $r$-bounded rank.
\end{proof}

\subsection{Completion of the proof of Theorem~\ref{main}}

Let $G$ be a finite group admitting a group of coprime automorphisms $A$ such that any subset of $I_G(A)$ generates an $r$-generator subgroup. We wish to prove that $[G,A]$ has $r$-bounded rank. First we show that we can pass to a subgroup of $r$-bounded index.

\begin{lemma}\label{l-g2} Suppose that $G$ has an $A$-invariant subgroup $H$ such that the index $|G:H|$ and the rank of $[H,A]$ are both $r$-bounded. Then the rank of $[G,A]$ is $r$-bounded.
\end{lemma}

\begin{proof} Replacing $H$ with $H\cap [G,A]$ we can assume that $G=[G,A]$. The normal core $\bigcap_{g\in G}H^g$ of $H$ is $A$-invariant and has index at most $|G:H|!$. Therefore we can assume from the outset that $H$ is a normal subgroup. Since the index $|G:H|$ is $r$-bounded, the number of conjugates of the subgroup $[H,A]$ is $r$-bounded. All these conjugates are normal subgroups of $H$ and have $r$-bounded rank and therefore their product, which is the normal closure $[H,A]^G$, has $r$-bounded rank. Hence it suffices to prove that the rank of $G/[H,A]^G$ is $r$-bounded. The image of $H$ in $G/[H,A]^G$ is centralized by $A$ and therefore is contained in the centre of $G/[H,A]^G$ by Lemma~\ref{l-ker}(a), since $G=[G,A]$ by assumption. Hence $G/[H,A]^G$ has centre of $r$-bounded index. By Schur's theorem \cite[Theorem 4.12]{rob} the derived subgroup of $G/[H,A]^G$ has $r$-bounded order. The rank of the abelian derived quotient of $G=[G,A]$ is $r$-bounded by the hypothesis that $I_G(A)$ generates an $r$-generator subgroup.
\end{proof}

\begin{lemma}\label{l-g3} The subgroup $[G,A]$ has an $A$-invariant soluble-by-semisimple-by-soluble subgroup $H$ of $r$-bounded index.
\end{lemma}

\begin{proof}
In the proof we can obviously assume that $G=[G,A]$ and the soluble radical of $G$ is trivial. By Lemma \ref{l-444} then $F^*(G)$ is a product of $r$-boundedly many simple factors. Then for the required subgroup $H$ we can take the kernel of the natural action of $G$ on the simple factors, which is semisimple-by-soluble and $A$-invariant and has $r$-bounded index. Indeed, since the centralizer of $F^*(G)$ is trivial, the group $G$ embeds into the automorphism group of $F^*(G)$. If $ F^*(G)=T_1\times \dots \times T_n$, where
 the $T_i$ are non-abelian simple groups, then $ \operatorname{Aut} F^*(G) = (\operatorname{Aut} T_1\times \dots \times \operatorname{Aut}T_n)U$,
 where $U$ is a subgroup of the symmetric group $S_n$. Since $n$ is $r$-bounded and the outer automorphism groups $ \operatorname{Aut}T_i/T_i$ are known to be soluble by the classification, the result follows.
\end{proof}

\begin{proof}[Proof of Theorem~\ref{main}] Recall that $G$ is a finite group admitting a group of coprime automorphisms $A$
such that any subset of $I_G(A)$ generates an $r$-generator subgroup. We need to prove that $[G,A]$ has $r$-bounded rank.

By Lemma~\ref{l-g3} the subgroup $[G,A]$ has an $A$-invariant soluble-by-semisimple-by-soluble subgroup $H$ of $r$-bounded index. By Lemma~\ref{l-g2} it is sufficient to prove that the rank of $[H,A]$ is $r$-bounded. Hence we can assume from the outset that $G=[G,A]$ is soluble-by-semisimple-by-soluble, so that we have a normal series $1\leqslant S\leqslant T\leqslant G$ in which both $G/T$ and $S$ are soluble, while $T/S$ is semisimple. By choosing $S$ to be the soluble radical, and then $T$ to be the inverse image of the generalized Fitting subgroup of $G/S$, we can assume that both $S$ and $T$ are $A$-invariant.
The rank of the semisimple section $T/S$ is $r$-bounded by Lemma~\ref{l-444}. Let $B=C_A(T/S)=\{a\in A\mid [T,a]\leqslant S\}$. Then $A/B$ has $r$-bounded order by Lemma~\ref{l-autsem}. By Lemma~\ref{l-ker}(a) the subgroup $[G,B]$ centralizes $T/S$ and therefore is soluble. Then this subgroup $[G,B]=[[G,B],B]$ has $r$-bounded rank by Proposition~\ref{p-sol}. The quotient $\bar G=G/[G,B]$ admits the action of the group $A/B$, which is, in particular, generated by $r$-boundedly many elements $a_1,\dots,a_{f(r)}$. By Corollary~\ref{c-cyclic} every subgroup $[\bar G, a_i]$ has $r$-bounded rank. Since
$$[\bar G,A/B]=\prod_{i=1}^{f(r)}[\bar G,a_i],$$
it follows that $\bar G=[\bar G,A/B]$ has $r$-bounded rank. Thus, both $[G,B]$ and $\bar G=G/[G,B]$ have $r$-bounded rank, whence $G$ has $r$-bounded rank.
\end{proof}

\section{Ranks in $p$-soluble groups}

\begin{proof}[Proof of Theorem~\ref{second}]
 Recall that $p$ is a prime, $G$ is a finite $p$-soluble group, and $P$ is a Sylow $p$-subgroup of $G$ such that any subgroup generated by elements of $I_G(P)$ is $r$-generator. We need to prove that $[G,P]$ has $r$-bounded rank. First we deal with the case where $G=P$.

\begin{lemma}\label{l-2pcase}
Suppose that $P$ is a finite $p$-group such that any subgroup generated by commutators of elements of $P$ is $r$-generator. Then $[P,P]$ has $r$-bounded rank.
 \end{lemma}

 \begin{proof}
 Let $M$ be a maximal normal abelian subgroup of $P$. Then $[M,P]$ has rank at most $r$ by hypothesis. Let $C=C_P([M,P])$. Then $[[C,M],C]=1$ and $[[M,C],C]=1$, whence $[[C,C],M]=1$ by the Three Subgroups Lemma.
 Since $C_P(M)=M$, we must have $[C,C]\leqslant M$, and therefore $C$ is a normal metabelian subgroup. The rank of $[C,P]$ is $r$-bounded. Indeed, $[C,C]$ is abelian and therefore has rank at most $r$. The quotient $C/[C,C]$ is abelian and therefore $[C,P]/[C,C]$ has rank at most $r$. Hence the rank of $[C,P]$ is at most $2r$.

 The quotient $P/C$ acts faithfully by automorphisms on $[M,P]$ and therefore has $r$-bounded rank by Lemma~\ref{l-gmh}.

 In the quotient $\bar P=P/[C,P]$ the image $\bar C$ of $C$ is central. The intersection $[\bar P, \bar P]\cap \bar C$, which is a subgroup of the Schur multiplier of $\bar P/\bar C$, has $r$-bounded rank by the Lubotzky--Mann theorem \cite[Theorem~4.2.3]{LM}. The result follows, since the ranks of $P/C$ and $[C,P]$ are $r$-bounded, as shown above.
 \end{proof}

We return to the proof of Theorem~\ref{second}. Since the rank of $[P,P]$ is $r$-bounded by Lemma~\ref{l-2pcase}, the normal closure $[P,P]^G$ has $r$-bounded $p$-length by Lemma~\ref{l-rpl}. The $p$-length of $G/ [P,P]^G$ is at most 1, since its Sylow $p$-subgroup is abelian. Hence the $p$-length of $G$ is also $r$-bounded.

It is convenient to define the parameter $l_{p\times p'}(G)$ as the minimum length of a normal series of $G$ all of whose factors $L$ are of the form $O_p(L)\times O_{p'}(L)$. We call such series $({p\times p'})$-series. This parameter $l_{p\times p'}(G)$ is clearly at most double the $p$-length of $G$ plus 1, and therefore is also $r$-bounded. Hence we can prove the theorem by induction on $l_{p\times p'}(G)$.

For $l_{p\times p'}(G)=1$ we have $G=P\times H$ (where $H=O_{p'}(G)$), whence $[G,P]=[P,P]$ and the result is furnished by Lemma~\ref{l-2pcase}.

Now let $l_{p\times p'}(G)\geqslant2$. In a $({p\times p'})$-series of length $l_{p\times p'}(G)\geqslant2$, let $N$ be the second term, which satisfies $l_{p\times p'}(N)=2$. Let $P_0$ be a Sylow $p$-subgroup of $N$ contained in $P$, and $H_0$ a Hall $p'$-subgroup of $N$.

\begin{lemma}\label{l-both}
Both $[O_p(N),H_0]$ and $[O_{p'}(N),P_0]$ have $r$-bounded rank.
\end{lemma}

\begin{proof}
The group $P_0$ acts by coprime automorphisms on $O_{p'}(N)$ and $I_{O_{p'}(N)}(P_0)\subseteq I_G(P)$. Therefore the hypothesis of Theorem~\ref{second} ensures that every subset of $I_{O_{p'}(N)}(P_0)$ generates an $r$-generator subgroup. Hence the rank of $[O_{p'}(N),P_0]$ is $r$-bounded by Theorem \ref{main}.

The group $H_0$ acts by coprime automorphisms on $O_{p}(N)$ and $I_{O_{p}(N)}(H_0)\subseteq I_G(P)$ because every commutator in $I_{O_{p}(N)}(H_0)$ involves an element of $O_{p}(N)\leqslant P$. Therefore every subset of $I_{O_{p}(N)}(H_0)$ generates an $r$-generator subgroup by the hypothesis of Theorem~\ref{second}. Hence the rank of $[O_{p}(N),H_0]$ is $r$-bounded by Theorem \ref{main}.
\end{proof}

We are now ready to finish the proof of Theorem~\ref{second}. Note that $[O_{p'}(N),P_0]$ and $[O_p(N),H_0]$ are normal subgroups of $G$. Indeed, $[O_{p'}(N), P_0]$ is normal in $O_{p'}(N)P_0$, which is normal in $G$. By the Frattini argument we have $G=O_{p'}(N)P_0N_G(P_0)$ and $N_G(P_0)$ also normalizes $[O_{p'}(N), P_0]$. Similarly, $[O_p(N), H_0]$ is normal in $O_p(N)H_0$, which is normal in~$G$. By the generalized Frattini argument for the Hall $p'$-subgroup $H_0$ of $O_p(N)H_0$, we have $G=O_p(N)H_0N_G(H_0)$ and $N_G(H_0)$ also normalizes $[O_p(N), H_0]$.

We set $R=[O_p(N),H_0]\cdot [O_{p'}(N),P_0]$. The image $\bar N$ of $N$ in the quotient $G/R$
clearly satisfies $\bar N=O_p(\bar N)\times O_{p'}(\bar N)$. Therefore the images of the terms of our $({p\times p'})$-series in $G/R$ form a $({p\times p'})$-series of smaller length. By the induction hypothesis, the rank of the image of $[G,P]$ in $G/R$
has $r$-bounded rank. Since the rank of $R=[O_p(N),H_0]\cdot [O_{p'}(N),P_0]$ is $r$-bounded by Lemma~\ref{l-both}, the result follows.
\end{proof}

\section{Counterexamples}\label{counterexamples}

In this section we produce examples showing that Theorem~\ref{second} is not valid without the assumption of $p$-solubility. Actually, for each prime $p$ we produce a series of finite groups~$G$ having a Sylow $p$-subgroup $P$ such that any subgroup generated by a subset of $I_G(P)$ can be generated by 3 elements, while the rank of $[G,P]$ is unbounded. We give slightly different examples for $p=2$ and $p\ne 2$.

We first give examples for $p=2$.

\begin{example}\label{ex1}
 Let $q$ be a prime such that $q \equiv 3 \mod 8$. Note that then we also have $q^e\equiv 3 \mod 8$ for any odd $e$. Let $G=PSL_2(q^e)$ for odd $e=1,3,5,\dots $. For these $G$, it is easy to see that a Sylow $2$-subgroup $P$ of $G$ is elementary abelian of order~$4$.

We claim that every subgroup generated by elements of $I_G(P)$ is generated
by at most $3$ elements. This suffices, since clearly the rank of $G$ is at least $e$ and $G=[G,P]$.

Let $H$ be a subgroup generated by elements of $I_G(P)$. Note that any subgroup
of $G$ not contained in a Borel subgroup is generated by at most two elements. Indeed, the subgroups of~$G$ are either contained in a Borel subgroup, or are dihedral, or are contained in a subgroup isomorphic to $S_4$ or $A_5$, or are isomorphic to $PSL_2(q^k)$ for some $k$.

So it remains to consider the case where $H \leqslant B = UT$, where $B$ is a Borel subgroup, $U$ is its unipotent radical, and $T$ is its maximal torus.

We first note that $U \cap I_G(P) =\{1\}$, since no involution inverts a nontrivial
element of~$U$, because $B$ has odd order when $q^e\equiv 3 \mod 8$.

If $1 \ne x \in P$, then $I_G(x) \cap B = T_x$, where $T_x$ is the (unique) maximal torus of $B$ that is inverted by $x$; actually, $T_x = B \cap B^x$.
Since $I_G(P)=I_G(x_1)\cup I_G(x_2)\cup I_G(x_3)$ for the three involutions $x_i\in P$, any subset of $B \cap I_G(P)$ generates a subgroup generated by elements of the three cyclic tori corresponding to each involution in $P$ and so is $3$-generated. Since there is no bound on the rank of $U$ for $q^e$ with $e$ increasing, we see that there is no bound on the rank of $G=[G,P]$.
\end{example}

Next we produce a similar example for $p\ne 2$. Here, the details are more complicated.

\begin{example}\label{ex2}
 Given an odd prime $p$, choose a prime $\ell \equiv (p-1) \mod p^2$
and let $e$ be any positive integer coprime to $2p$. Let $q= \ell^e$ and consider
$G=SL_2(q)$. Note that a Sylow $p$-subgroup $P$ of $G$ has order $p$
and $P \cap B =1$ for any Borel subgroup $B$. Fix an element $a \in G$ of order $p$.

As noted in Example~\ref{ex1}, any subgroup
of $G$ not contained in a Borel subgroup is generated by at most two elements.
Therefore it suffices to show that any subgroup generated by
elements of $I_G(P) \cap B$ is generated by two elements.

Let $t $ be any noncentral semisimple element conjugate to an element of $B$. Consider the set of
triples $(x,y,z)$ of elements of $G$ with $xyz=1$ such that $x, y$ are conjugate to $a$ (which is conjugate to $a^{-1}$) and $z$ is conjugate to $t$. This set is nonempty by \cite{Gow, GT}. Since the centralizer of $x$ in $PGL_2(q)$ is a torus of order
$q+1$ and the centralizer of $t$ is a torus of order $q-1$, the centralizer
of any conjugate of $a$ and any conjugate of $t$ is trivial (in $PGL_2(q))$. Clearly, any two conjugates
of $a$ and $t$ generate
an irreducible subgroup and so an absolutely irreducible subgroup.
It follows by \cite[Theorem~2.3]{SV} that all such triples form a regular $PGL_2(q)$-orbit.

If we only consider such triples with $x=a^{-1}$, they form a single orbit under $C$, the centralizer
of $a$ (in $PGL_2(q))$. So there are $q+1$ triples $(a^{-1},y,z)$ with $a^{-1}yz=1$ such that $y$ is conjugate to $a$ and
$z$ is conjugate to $t$.
Note that for such triples we have $z=t^u=(a^{-1}a^v)^{-1}=[v,a]\in I_G(a)$ for some $u,v\in G$. Since $C$ acts regularly on
the set of Borel subgroups, we obtain $|B \cap I_G(a) \cap t^G| =1$. Thus there is a unique conjugate $s$ of $t$ with $s=a^{-1}a^v\in I_G(a)$. By uniqueness it follows that if $s \in I_G(a)$ is conjugate to an element of
$SL_2(\ell^f)$ for some $f$ dividing $e$, then $s \in SL_2(\ell^f) \cap B$.

If we replace $t$ with a nontrivial unipotent element, in fact there are no solutions \cite{gu1}.
The only other noncentral class is the class of $-u$ for $u$ a nontrivial unipotent element.
The argument above applies to that class as well and so $|I_G(a) \cap B \cap(-u)^G|\leqslant 1$
and $I_G(a) \cap B \cap (-u)^G \subset SL_2(\ell)$.
The same argument applies to any power of $a$. Note that since all Borel subgroups are
conjugate under the centralizer of $a$ in $PGL_2(q)$, it suffices to work in a fixed
Borel subgroup $B$ which may take to be defined over the prime field.
Write $B=TU$,
where $T$ is a maximal torus of $B$, which is cyclic of order $q-1$, and $U$ is
 the unipotent radical of $B$. Let $T = \langle w \rangle$.

Consider some subset $S$ of $I_G(P) \cap B$. Write the nontrivial elements as $w^ju$
with $1 \leqslant j < q-1$ and $u \in U$ and consider $\langle S \rangle$. We claim
that $\langle S \rangle$ can be generated by two elements. We induct on $e$.
If $e=1$, every subgroup of $B$ is either cyclic or $2$-generated.

If all $w^j \in SL_2(\ell^f)$ for some $f<e$, then by the argument above
$\langle S \rangle \leqslant B \cap SL_2(\ell^f)$ and the result follows by induction.
Otherwise, $\langle S \rangle U/U$ acts irreducibly on $U$ and in particular $U$
is a cyclic $\langle S \rangle$-module. Since $\langle S \rangle U/U$ is cyclic,
it follows that $\langle S \rangle$ is $2$-generated.

Thus, we have shown that any subgroup of $G$ generated by elements of $I_G(P)$
can be generated by at most two elements. On the other hand, $e$ can be
arbitrarily large and the rank of $U$ is at least $e$. Thus, we see that there is no bound on the rank of $G=[G,P]$.
\end{example}

\section{Proofs of Theorems~\ref{third} and \ref{fourth}}

We begin with a useful lemma used in the proofs of both Theorems~\ref{third} and \ref{fourth}.

\begin{lemma}\label{l-focal}
Suppose that $G$ is a finite group, $p$ is a fixed prime, $P$ is a Sylow $p$-subgroup of $G$, and $r$ is a positive integer such that any subgroup generated by a subset of $I_G(P)$ is $r$-generated. Then the rank of $P\cap G'$ is $r$-bounded.
\end{lemma}

\begin{proof}
 The subgroup $P\cap G'$ is generated by a subset of $I_G(P)$ by D.~Higman's focal subgroup theorem \cite{dhig} (see also \cite[Theorem~7.3.4]{gore}). Hence $P\cap G'$ is generated by $r$ elements, whence the abelian quotient $(P\cap G')/P'$ has rank at most $r$. Since the rank of $P'$ is $r$-bounded by Lemma~\ref{l-2pcase}, the result follows.
\end{proof}

First we give a short proof of Theorem~\ref{fourth}.

\begin{proof}[Proof of Theorem~\ref{fourth}] Recall the hypotheses of Theorem~\ref{fourth}: $G$ is a finite group such that
for every prime $p$ dividing $|G|$, for any Sylow $p$-subgroup $P$, any subgroup generated by a subset of $I_G(P)$ can be generated by $r$ elements. We need to prove that the derived subgroup $G'$ has $r$-bounded rank. The rank of a finite group is known to be at most $s+1$, where $s$ is the maximum of the ranks of its Sylow subgroups \cite{gu2, lo-ma, Lu}. The ranks of all Sylow subgroups of $G'$ are $r$-bounded by Lemma~\ref{l-focal}; hence the result.
\end{proof}

We now recall the hypotheses of Theorem~\ref{third}: $G$ is a finite group, $p$ is a fixed prime, $P$ is a Sylow $p$-subgroup of $G$, and $r$ is a positive integer such that
\begin{enumerate}
\item[\rm (1)] any subgroup generated by a subset of $I_G(P)$ is $r$-generated; and
\item[\rm (2)] for any $x$ in $I_G(P)$, any subgroup generated by a subset of $I_G(x)$ is $r$-generated.
\end{enumerate}
We fix these notations and assume these conditions for the rest of the section. Our aim is to prove that the rank of $[G,P]$ is $r$-bounded.

First we consider the case of simple groups.

\begin{proposition}\label{p-simple}
If $G$ is simple, then the rank of $G$ is $r$-bounded.
\end{proposition}
\reversemarginpar
\begin{proof}
Of course, there is nothing to prove for sporadic groups. We now consider alternating groups.

\begin{lemma}\label{l-an}
If $G=A_n$ is an alternating group on $n$ symbols, then $n$ is $r$-bounded.
\end{lemma}

\begin{proof}
 Let $(12\dots p)\in P$. The commutator $[(12\dots p), (123)]=(124)$ belongs to $I_G(P)$, and so do the commutators
 $$
 [(12\dots p),\; (4k+1\; 4k+2\; 4k+3)]=(4k+1\;4k+2\; 4k+4)
 $$
 for $k=1,2,\dots ,\lfloor p/4\rfloor-1$. Since these 3-cycles are independent, they generate an elementary abelian 3-group of rank $\lfloor p/4\rfloor$, whence $p$ is $r$-bounded by condition (1). The same procedure repeated for $\lfloor n/p\rfloor$ independent $p$-cycles produces an abelian 3-group of rank $\lfloor p/4\rfloor\cdot \lfloor n/p\rfloor$ generated by elements of $I_G(P)$, whence $n$ is $r$-bounded by condition (1).
\end{proof}

Thus, it suffices to consider finite simple groups of Lie type of Lie rank $d$ defined over the field of $\ell^e$ elements for some prime $\ell$. We first note that $d$ must be $r$-bounded, and this fact will also be used later.

\begin{lemma}\label{l-3lierank}
If $G$ is a finite simple group of Lie type of Lie rank $d$, then $d$ is $r$-bounded.
\end{lemma}

\begin{proof} We may assume $d >8$ and so $G$ is a classical group
with natural module $V$ of dimension $n$. It is well-known that $n=d+1$, or $n=2d$, or $n=2d+1$. We shall show that a large alternating group is always contained in $G$ and so we may apply the previous result.

Note that for $m \geqslant 5$, the group $A_m$ has an irreducible self-dual representation of dimension at most $m-1$ defined
over the prime field. Thus, $A_{m}$ embeds in $SL_{m-1}(q)$ and $SU_{m-1}(q)$. Since $GL_n(q)$ embeds in $Sp_{2n}(q)$, we see that $A_{n+1}$ embeds in $Sp_{2n}(q)$.

If $q$ is odd, then this irreducible representation is always orthogonal (since the
permutation module is orthogonal) and so, adjusting for the type if necessary,
we see that $A_n$ embeds in $SO^{\pm}_{n}(q)$.

If $q$ is even, then $A_{2n}$ embeds in $Sp_{2n-2}(q)$, and this group embeds in
either of the orthogonal groups $SO^{\pm}_{2n}(q)$ (as the derived subgroup of the stabilizer of a nondegenerate
vector).

Thus, we see that for $d>8$ the group $G$ contains an alternating group of degree at least~$d$. The result follows by Lemma~\ref{l-an}.
\end{proof}

We also note the following well-known fact.

\begin{lemma}\label{l-rlt}
 If $G$ is a finite simple group of Lie type of Lie rank $d$ over the field of
size $q = \ell^e$ for a prime $\ell$, then the (Pr\"ufer) rank of $G$ is bounded in terms of $d$ and $e$.
\end{lemma}

In particular, this upper bound for the rank of $G$ is independent of $\ell$.

\begin{proof}
To bound the rank of $G$, it suffices to bound the
rank of each Sylow $s$-subgroup of~$G$ \cite{gu2, lo-ma, Lu}. If $s \ne \ell$, it is well known that the rank of a Sylow $s$-subgroup of $G$ is at most the (untwisted)
Lie rank of $G$ plus the rank of the Weyl group \cite[Sec.~4.10]{GLS3}, which is bounded in terms of $d$. The Sylow $\ell$-subgroup has order less than $\ell^{ed^2}$ and its rank is at most
$ed^2$.
\end{proof}

We proceed with the proof of Proposition~\ref{p-simple}. By Lemma~\ref{l-an}, we may assume $G$ to be a finite simple group of Lie type of Lie rank $d$ over the field of $q$ elements with $q = \ell^e$ for a prime $\ell$, and we may assume that $d$ is fixed and indeed assume the type of $G$ is fixed. By Lemma~\ref{l-rlt} it remains
to show that $e$ is bounded in terms of $r$.

For most cases, we only require condition~(1), although
we cannot do this in all cases as Examples~\ref{ex1} and \ref{ex2} show.

If $p = \ell$, then $P/\Phi(P)$ is elementary abelian of order at least $\ell^e$. Since $P = P \cap G'$ (as $G$ is simple) and the rank of $P=P\cap G'$ is $r$-bounded by Lemma~\ref{l-focal}, it follows that $e$ is $r$-bounded.

So we may assume that $p \ne \ell$. First suppose that
$p$ divides the order of some maximal
parabolic subgroup $M$. Let $M = LQ$, where $L$ is a Levi subgroup of $M$
and $Q$ is its unipotent radical. Recall that $C_M(Q) \leqslant Q$ (that is, all parabolic subgroups
are $\ell$-constrained). If $x$ is an element of order $p$ in $M$, we see that $x$
acts nontrivially on $Q$ and so also on $Q/\Phi(Q)$ as a linear transformation over
the field of $q=\ell^e$ elements. Hence $[x, Q/\Phi(Q)]$ has rank at least~$e$. Since $[x, Q/\Phi(Q)]$ is generated by the images of elements of $I_G(x)$ and $x$ is conjugate to an element of $P$, we obtain $ e \leqslant r$ by condition~(1).

Therefore we may assume that $p$ does not divide the order of any parabolic subgroup and so every nontrivial element of $P$ is semisimple regular. By \cite[Corollary]{Gow}, then $I_G(P)$ intersects every nontrivial conjugacy class of semisimple elements.

 By the above, we may assume that there exists a semisimple regular element $x$ of prime order in a Borel subgroup $B$ such that $x \in I_G(P)$.
We now apply to $I_G(x)$ the arguments as in the above case where an element of order $p\ne \ell$ is contained in a parabolic subgroup. Using condition~(2) we obtain
that $e \leqslant r$.
\end{proof}

\begin{proof}[Completion of the proof of Theorem~\ref{third}] First we extend Proposition~\ref{p-simple} from simple groups to semisimple.

\begin{lemma}\label{l-3semi}
If $G$ is a direct product of $m$ non-abelian simple groups of order divisible by~$p$, then $m\leqslant r$ and the rank of $G$ is $r$-bounded.
\end{lemma}

 \begin{proof}
 Let $P$ be the product of Sylow $p$-subgroups $P_i$ of the simple factors $G_i$. We know that $P_i$ is not in the centre of the normalizer $N_{G_i}(P_i)$ by Burnside's theorem on normal $p$-complements. Hence each factor $G_i$ contains a non-trivial commutator $[x_i,n_i]\in I_G(P)\cap P_i$ for $x_i\in P_i$. These commutators generate an abelian $p$-subgroup of rank $m$. Hence $m\leqslant r$ by condition (1). Then the rank of $G$ is $r$-bounded, since the rank of each factor $G_i$ is $r$-bounded by Proposition~\ref{p-simple}.
 \end{proof}

 When $G$ is a $p$-soluble group, the result already follows from Theorem~\ref{second}. Therefore we assume henceforth that $G$ is not $p$-soluble.

\begin{lemma}\label{l-3derived}
If the quotient $G/R$ of $G$ by the $p$-soluble radical $R$ is a direct product of non-abelian simple groups of order divisible by $p$, then the rank of $[G,P]$ is $r$-bounded.
\end{lemma}

\begin{proof}
The quotient $G/R$ is a direct product of at most $r$ finite simple groups by Lemma~\ref{l-3semi}. By Lemma~\ref{l-an} the alternating factors have $r$-bounded order, as obviously do sporadic factors. By Lemma~\ref{l-3lierank} the factors of Lie type have $r$-bounded Lie ranks. By the result of J.~Hall, M.~Liebeck, and G.~Seitz \cite{hls} (later improved by R.~Guralnick and J.~Saxl \cite{GSaxl}), each of these factors can be generated by $r$-boundedly many $p$-elements. Altogether, the quotient $G/R$ is generated by $r$-boundedly many $p$-elements.

We choose $p$-elements $a_1,\dots ,a_m$ of $G$ that generate $G$ modulo $R$, where $m$ is $r$-bounded, and set $H=\langle a_1,\dots ,a_m\rangle$. Each $a_i$ has a conjugate $a_i^{g_i}\in P$, for which $[R,a_i^{g_i}]$ has $r$-bounded rank by Theorem~\ref{second} applied to $RP$. Then $[R,a_i]=[R,a_i^{g_i}]^{g_i^{-1}}$ also has $r$-bounded rank. As a result, the product $\prod_{i}[R,a_i]$ has $r$-bounded rank, since the number of factors is $r$-bounded.

Note that the product $\prod_{i}[R,a_i]$ is normal in $G=RH$. Indeed, it is normal in $R$, since each $[R,a_i]$ is normal in $R$, and this product is $H$-invariant because it is the smallest normal subgroup of $R$ such that $H$ acts trivially on the quotient.

Since the rank of $\prod_{i}[R,a_i]$ is $r$-bounded, we can pass to the quotient $G/\prod_{i}[R,a_i]$ and assume that $[H,R]=1$.

Let $H_p$ be a Sylow $p$-subgroup of $H$; then $H_p^g\leqslant P$ for some $g\in G$. If $g=xy$, where $x\in R$ and $y\in H$, we see that $H_p^g=H_p^{xy}=H_p^y\leqslant H$, since $x$ centralizes $H$ by our assumption. Thus, $H\cap P$ is a Sylow $p$-subgroup of $H$. Note that $P=(P\cap R)\cdot (P\cap H)$.

We need to bound the rank of $[G,P]=[RH,P]=[R,P]\cdot [H,P]$. The rank of $[R,P]$ is $r$-bounded by Theorem~\ref{second} applied to $RP$. Since $[R,P]\leqslant R$ is centralized by $H$ by assumption, it remains to bound the rank of $[H,P]=[H,\,(P\cap R)\cdot (P\cap H)]=[H,\,P\cap H]$.

In fact, the rank of $[H,H]$ is $r$-bounded.
Since $[H,R]=1$, the intersection $[H,H]\cap R$ is central in $H$ and therefore is isomorphic to a subgroup of the Schur multiplier of $H/(H\cap R)$. Since the rank of $H/(H\cap R)\cong G/R$ is $r$-bounded by Lemma~\ref{l-3semi}, the rank of this Schur multiplier is also $r$-bounded by the Lubotzky--Mann theorem \cite[Theorem~4.2.3]{LM}. As a result, the rank of $[H,H]$ is $r$-bounded, and so is the rank of $[H,P]$, as required.
\end{proof}

We now show that in the general case the group $G$ has a subgroup of $r$-bounded index that satisfies the hypotheses of Lemma~\ref{l-3derived}.

\begin{lemma}\label{l-3lambda1}
The group $G$ has a normal subgroup $M$ of $r$-bounded index containing the $p$-soluble radical $R$ of $G$ such that $M/R$ is a direct product of non-abelian simple groups of order divisible by~$p$.
\end{lemma}

\begin{proof}
The socle $F^*(G/R)$ of $G/R$ is a direct product of non-abelian simple groups of order divisible by~$p$. The number of these factors is at most $r$ by Lemma~\ref{l-3semi}. We take the inverse image of $F^*(G/R)$ in $G$ for the required subgroup $M$. We only need to show that the index of $M$ is $r$-bounded. Indeed, the group $G/R$ embeds into the automorphism group of $F^*(G/R)$. Let $ F^*(G/R)=T_1\times \dots \times T_n$, where the $T_i$ are non-abelian simple groups. Then $ \operatorname{Aut} F^*(G/R) = (\operatorname{Aut} T_1\times \dots \times \operatorname{Aut}T_n)U$,
 where $U$ is a subgroup of the symmetric group $S_n$. Since $n\leqslant r$, it remains to prove that the orders of the outer automorphism groups $ \operatorname{Aut}T_i/T_i$ are $r$-bounded. This is obviously true for sporadic factors. The alternating factors have $r$-bounded orders by Lemma~\ref{l-an}.
 The factors of Lie type have $r$-bounded Lie ranks and are defined over fields of orders $\ell ^e$ for primes $\ell$ with $r$-bounded exponents $e$ by Proposition~\ref{p-simple} and Lemma~\ref{l-3lierank}. Therefore their outer automorphism groups also have $r$-bounded orders.
 \end{proof}

We now finish the proof of Theorem~\ref{third}. Let $ M$ be the normal subgroup of $G$ of $r$-bounded index given by Lemma~\ref{l-3lambda1}, which is actually the inverse image of the socle $F^*(G/R)$ of the quotient $G/R$ by the $p$-soluble radical $R$. By Lemma~\ref{l-3derived} the rank of $[M,M\cap P]$ is $r$-bounded. Note that $M=R\cdot [M,M\cap P]$, so that $M/[M,M\cap P]$ is $p$-soluble. Since $M$ has $r$-bounded index in $G$, the normal closure $N$ of $[M,M\cap P]$ in $G$ is a product of $r$-boundedly many conjugates of $[M,M\cap P]$ and therefore also has $r$-bounded rank. Note that $M/N$ is $p$-soluble. Passing to the quotient $G/N$ we can assume that $N=1$.

Thus, we can assume that $G$ has a $p$-soluble normal subgroup $M$ of $r$-bounded index. By Theorem~\ref{second} applied to the product $MP$ the rank of $[M,P]$ is $r$-bounded. Since $M$ has $r$-bounded index in $G$, the normal closure $T$ of $[M,P]$ in $G$ is a product of $r$-boundedly many conjugates of $[M,P]$ and therefore also has $r$-bounded rank. Passing to the quotient $G/T$ we can assume that $T=1$, so that $[M,P]=1$.

Hence the index of the centralizer $C_G(P)$ is $r$-bounded. This centralizer contains a normal subgroup $C$ of $r$-bounded index. Then $C$ centralizes $[G,P]$ by Lemma~\ref{l-ker}(a). Since $C\cap [G,P]\leqslant Z([G,P])$, we obtain that the centre of $[G,P]$ has $r$-bounded index. By Schur's theorem
(see \cite[Theorem 4.12]{rob}) the derived subgroup $[G,P]'$ of $[G,P]$ has $r$-bounded order. Passing to the quotient $G/[G,P]'$ we can assume without loss of generality
that $[G,P]$ is abelian. Since $[G, P] =\langle I_G(P)\rangle$ is $r$-generated, the rank of $[G, P]$ is at most $r$.
\end{proof}

\end{document}